\documentclass[10pt, bezier,amstex, oneside,reqno]{amsproc}
\usepackage{amsmath, amssymb, amsthm, verbatim, euscript, amscd, textcomp, mathrsfs}
\usepackage[dvips]{graphicx}
\usepackage{epsfig}
 \usepackage[all,cmtip]{xy}
\usepackage{color}
 \usepackage[fleqn]{cases}
 \usepackage{accents}
\newlength{\dhatheight}
\newcommand{\doublehat}[1]{%
    \settoheight{\dhatheight}{\ensuremath{\hat{#1}}}%
    \addtolength{\dhatheight}{-0.35ex}%
    \hat{\vphantom{\rule{1pt}{\dhatheight}}%
    \smash{\hat{#1}}}}

\def\ie{\emph{i.e., }}
\def\eg{\emph{e.g., }}

\def\R{\mathbb R}
\def\Z{\mathbb Z}
\def\D{\mathbb D}
\def\DD{\tilde {\mathbb D}}
\def\U{\mathcal U}
\def\V{\mathcal V}

\def\C{\mathbb C}
\def\T{\mathbb T}

\def\AA{\EuScript A}

\def\ZZ{\mathcal Z}

\def\e{\varepsilon}

\newtheorem*{theoremM}{Main Theorem}
\newtheorem{theorem}{Theorem}[section]

\newtheorem{claim}[theorem]{Claim}

\newtheorem{lemma}[theorem]{Lemma}
 \theoremstyle{remark}
\newtheorem{remark}[theorem]{Remark}
\newtheorem{example}[theorem]{Example}

\theoremstyle{remark}

\renewcommand\S{\mathbb S}

\begin{document}
\author{Andrey Gogolev$^\ast$}
\title[Blow-ups of partially hyperbolic dynamical systems]{Surgery for partially hyperbolic dynamical systems I. \\
Blow-ups of invariant submanifolds. }
\thanks{$^\ast$The author was partially supported by NSF grant DMS-1204943}
\begin{abstract}
We suggest a method to construct new examples of partially hyperbolic diffeomorphisms. We begin with a partially hyperbolic diffeomorphism $f\colon M\to M$ which leaves invariant a submanifold $N\subset M$. We assume that $N$ is an Anosov submanifold for $f$, that is, the restriction $f|_N$ is an Anosov diffeomorphism and the center distribution is transverse to $TN\subset TM$. By replacing each point in $N$ with the projective space (real or complex) of lines normal to $N$ we obtain the blow-up $\hat M$. Replacing $M$ with $\hat M$ amounts to a surgery on the neighborhood of $N$ which alters the topology of the manifold. The diffeomorphism $f$ induces a canonical diffeomorphism $\hat f\colon \hat M\to \hat M$. We prove that under certain assumptions on the local dynamics of $f$ at $N$ the diffeomorphism $\hat f$ is also partially hyperbolic. We also present some modifications such as the connected sum construction which allows to ``paste together" two partially hyperbolic diffeomorphisms to obtain a new one. Finally, we present several examples to which our results apply.
\end{abstract}
\date{}
 \maketitle

\section{Introduction}
Let $M$ be a closed manifold. A diffeomorphism $f\colon M \to M$ is {\it partially hyperbolic} 
if the tangent bundle $TM$ splits into $Df$-invariant continuous subbundles 
$TM= E^{s}\oplus E^c \oplus E^{u}$ such that 
\begin{equation}
\label{def_ph}
 \|Df(v^s) \| <\lambda< \|Df (v^c)\| < \mu< \|Df(v^u) \|
\end{equation}
for some Riemannian metric $\|\cdot\|$, some $\lambda<1<\mu$ and all unit vectors $v^s\in E^s$, $v^c\in E^c$ and $v^u\in E^u$.

Similarly a flow $\varphi^t\colon M \to M$ is {\it partially hyperbolic} 
if the tangent bundle $TM$ splits into $Df$-invariant continuous subbundles 
$TM= E^{s}\oplus E^c \oplus E^{u}$ such that 
\begin{equation}
\label{def_ph_flow}
 \|D\varphi^t(v^s) \| <\lambda^t< \|D\varphi^t (v^c)\| < \mu^t< \|D\varphi^t(v^u) \|,\; t\ge 1,
\end{equation}
for some Riemannian metric $\|\cdot\|$, some $\lambda<1<\mu$ and all unit vectors $v^s\in E^s$, $v^c\in E^c$ and $v^u\in E^u$.

Partial hyperbolicity was introduced into smooth dynamics by Hirsch-Pugh-Shub~\cite{HPS} and by Brin-Pesin~\cite{BP} (motivated by a paper of Sacksteder~\cite{sack}). The importance of these definitions is well justified by the deep connections of partial hyperbolicity to stable ergodicity and robust transitivity. The discussions on stable ergodicity and robust transitivity and the original references can be found in recent surveys~\cite{HHU2, HP2, CHHU, HP}.

Examples of partially hyperbolic dynamical systems can be roughly classified (up to homotopy, finite iterates and finite covers) into the following (overlapping) classes:
\begin{enumerate}
\item Algebraic examples induced by affine diffeomorphisms of Lie groups;
\item Geodesic flows in negative curvature;
\item Skew products with slow dynamics in the fiber and partially hyperbolic dynamics in the base;
\item Surgery examples;
\item Skew products with Anosov (or partially hyperbolic) dynamics in the fiber and slow dynamics in the base (fiberwise Anosov);
\item Twisting of Anosov flows.
\end{enumerate}
The first three classes of examples are classical and a lot of research in the past decades was focused on these examples. Some of the algebraic examples can be viewed as fiberwise Anosov (class 5). Recently, it was demonstrated that this class also contains some non-algebraic examples~\cite{GORH}. Even more recently, new examples (the last class 6) were discovered by composing the existing examples (such as time one maps of Anosov flows) with homotopically non-trivial diffeomorphisms which respect cone fields, see~\cite[Section 5]{HP} for an overview and references therein.

As outlined in the abstract, the current paper makes a contribution to the surgery constructions of partially hyperbolic diffeomorphisms. First surgery constructions of Anosov flows were discovered by Franks-Williams~\cite{FW}  and by Handel-Thurston~\cite{HT}. Since then many more 3-dimensional Anosov flows were constructed by using surgery. The approach used in these surgery constructions is to make ``hyperbolic pieces'' by cutting the ambient manifold of a known example along well-positioned (\eg transverse to the flow) codimension one submanifolds and then create new examples by assembling the ``hyperbolic pieces'' in various ways. For a long time this type of constructions were restricted to the realm of 3-dimensional Anosov flows, but recently the cut-and-paste approach have spread out into the classification program of 3-dimensional partially hyperbolic diffeomorphisms as well as to higher dimensions.

Surgery constructions here are quite different because we make use of the Anosov submanifold (which is also well-positioned with respect to the dynamics, but is not of codimension one) which is tangent to the stable and unstable distributions and works equally well for diffeomorphisms and for flows. The examples which we work out in this paper all belong to the class of fiberwise Anosov partially hyperbolic dynamical systems. This new pool of examples vastly expands this class of fiberwise Anosov partially hyperbolic  dynamical systems. We plan to further develop the blow-up approach and produce more examples, some of which are not fiberwise Anosov.

We are not aware of any prior appearance of blow-ups in partially hyperbolic dynamics. However, blow-ups have been known to be a useful construction tool in dynamics for a long time. At least, it goes back to work of Denjoy~\cite{D}, where he used one dimensional blow-up of an orbit to give an example of non-transitive circle diffeomorphism with an irrational rotation number. Katok~\cite{K} used the blow-up of a fixed point in his construction of Bernoulli diffeomorphism of $\D^2$ in order to pass from $\S^2$ to $\D^2$; also Katok-Lewis~\cite{KL} used the blow-up of a fixed point to produce examples of non-standard actions of $SL(n,\Z)$.

\section{The Main Theorem}
\subsection{Dominant Anosov submanifolds}
Let $f\colon M\to M$ be a partially hyperbolic diffeomorphism with an invariant splitting $TM=E^s\oplus E^c\oplus E^u$ controlled by $\lambda<1<\mu$ as in~(\ref{def_ph}). An invariant submanifold $N\subset M$ is called {\it Anosov} if
$$
TN=E^s\oplus E^u.
$$
Further an Anosov submanifold $N$ is called {\it dominant}\footnote{The domination condition is analogous to the well-known ``center-bunching" condition on the center distribution. We use a different term here because we view domination as a property of the fast distributions rather than the center.} if for all $x\in N$ and all unit vectors $v^c\in E^c(x)$
\begin{equation}
\label{eq_dominant}
{\lambda'}\le \|Df v^c\|\le \mu'\;\;\;\mbox{with}\;\;\;\; \frac{\lambda'}{\mu'}> \max(\lambda, \mu^{-1})
\end{equation}
An important special case is when $\mu^{-1}=\lambda$ and the domination inequality is
\begin{equation}
\label{eq_dominant2}
\sqrt{\lambda}\le \|Df v^c\|\le \sqrt{\lambda^{-1}}
\end{equation}
We proceed to impose a strong assumption on local dynamics at $N$. Namely, we will assume that the dynamics in the neighborhood of $N$ is {\it locally fiberwise}. That means that a neighborhood of $N$ can be smoothly identified with $\D^k\times N$, where $\D^k=\{ x\in\R^k: \;\|x\|<1\}$, so that the dynamics $f|_{\D^k\times N}$ is the product
\begin{equation}
\label{eq_fiberwise}
f(x,y)=(Ax, f_N(y)),\;\;\; (x,y)\in \D^k\times N\cap f^{-1}(\D^k\times N),
\end{equation}
where $f_N$ is the Anosov map given by the restriction $f|_N$ and $A\colon\R^k\to\R^k$ is a hyperbolic linear map. Moreover, we assume that the distribution $E^s\oplus E^u$ is integrable on $\D^k\times N$ and is tangent to the $N$-fibers; that is, for all $(x,y)\in \D^k\times N$ we have 
\begin{equation}
\label{eq_fiberwise2}
Di_xT_yN=E^s\oplus E^u(x,y),
\end{equation}
 where $i_x\colon N\to \D^k\times N$ is given by $i_x(y)=(x,y)$.

Note that locally fiberwise condition implies, in particular, that the normal bundle of $N$ is trivial.

Similarly, we can  define dominant Anosov submanifold $N\subset M$ for a partially hyperbolic flow $\varphi^t\colon M\to M$. In the flow setting, the formula~(\ref{eq_fiberwise}) becomes
$$
\varphi^t(x,y)=(A^t(x), \varphi^t_N(y)),
$$
where $\varphi^t_N$ is identified with $\varphi^t|_N$ and $A^t\colon \R^k\to\R^k$ is a hyperbolic linear flow. The condition~(\ref{eq_fiberwise2}) becomes
$$
Di_xT_yN\subset E^s \oplus E^u(x,y).
$$

\begin{remark} 
\label{rem_center}
The restriction $E^c|_N$ is a ``horizontal" subbundle in the $(x,y)$-coordinates, because it is the only $Df$-invariant subbundle which is transverse to $TN$. Therefore, given the local form~(\ref{eq_fiberwise}), one can determine whether the submanifold $N\subset M$ is dominant by looking at the eigenvalues of $A$.
\end{remark}

\begin{remark} In this paper the locally fiberwise condition is viewed as a feature which makes proving our results an easier task. One can also view it as a bug which  crashes some potential applications.
\end{remark}

\begin{remark}
Existence of an Anosov submanifold is an obstruction to accessibility property of $f$. And the important role of the Anosov tori for 3-dimensional partially hyperbolic diffeomorphisms $f\colon M\to M$ was revealed in~\cite{HHU}. Rodriguez Hertz-Rodriguez Hertz-Ures conjecture that absence of Anosov tori implies ergodicity of a partially hyperbolic diffeomorphism $f\colon M^3\to M^3$. In the case when $M$ is a nilmanifold ($\neq\T^3$)  they verified this conjecture~\cite{HHU}.
\end{remark}

\subsection{The blow-up of an Anosov submanifold}
\label{sec_setup}
We begin by blowing up the disk $\D^k$ at the origin $0$. This amounts to replacing $0$ with the space of lines which pass through $0$. More precisely, the disk $\D^k$ is being replaced with the following subspace of $\D^k\times\R P^{k-1}$
\begin{equation}
\label{eq_D}
\tilde \D^k =\{(x, \ell(x)):\, x\in\D^k, x\in \ell(x)\},
\end{equation}
where $\ell(x)$ are lines passing though the $0$ and $x$. Then $\pi\colon\tilde\D^k\to\D^k$ given by $(x,\ell(x))\mapsto x$ collapses the projective space $\R P^{k-1}$ to $0\in\D^k$ and is one-to-one otherwise. It is easy to see that $\tilde \D^k$ is diffeomorphic to the connected sum $\D^k\#\R P^k$.

Now, by  taking the product with $N$, we obtain the blow-up $\tilde \D^k\times N\to\D^k\times N$ and then use the identity map to extend to the map $\pi\colon\hat M\to M$, which we still denote by $\pi\colon\hat M\to M$. By construction, $\pi$ collapses $\R P^{k-1}\times N$ to $N$ and is one-to-one otherwise. We will call $\R P^{k-1}\times N\subset\hat M$ {\it the exceptional set.}

Now let $A\colon\R^k\to\R^k$ be a linear map. Then, by linearity, $x\in\ell(x)$ if and only if $Ax\in A(\ell(x))$ and, hence, the formula $(x,\ell(x))\mapsto (x, A(\ell(x)))$ defines a diffeomorphism $\tilde A\colon\tilde\R^k\to\tilde\R^k$ of the blown-up $\R^k$, which we then restrict to $\tilde \D^k$.

Now, assuming that a partially hyperbolic diffeomorphism $f\colon M\to M$ is locally fiberwise at $N\subset M$, we define $\hat f\colon\tilde \D^k\times N\to\tilde \D^k\times N$ by
$$
\hat f\colon (x,y)\mapsto (\tilde A(x), f_N(y))
$$
and extend $\hat f$ to the rest of $\hat M$ using $f$. We conclude that if dynamics of $f$ is locally fiberwise in a neighborhood of $f$ then there is a canonical diffeomorphism $\hat f\colon\hat M\to\hat M$ which fits into the commutative diagram
\begin{equation}
\label{eq_blowup}
\xymatrix{
\hat M\ar_\pi[d]\ar^{\hat f}[r] & \hat M\ar_\pi[d]\\
M\ar^f[r] & M
}
\end{equation}
\begin{remark}
Note that, by construction, $\hat M$ can be obtained from $M$ through the following surgical procedure: remove the open set $\D^k\times N$ from $N$ and then replace it with $\tilde\D^k\times N$. In general, such surgery affects the algebraic topology of the underlying manifold.
 \end{remark}
 \begin{remark} To obtain the diagram~(\ref{eq_blowup}) one only needs to have an $f$-invariant submanifold $N$, see, \eg~\cite{stark}.
 \end{remark}
Analogous discussion (which we omit) in the continuous time setting yields the blown-up flow $\hat\varphi^t\colon\hat M\to \hat M$. 
Now we are ready to state our main result.

\begin{theoremM} Let $f\colon M\to M$ ($\varphi^t\colon M\to M$) be a partially hyperbolic diffeomorphism (flow) and let $N\subset M$ be an invariant, dominant, Anosov submanifold of $M$. Also assume that the dynamics is locally fiberwise in a neighborhood of $N$. Let $\pi\colon\hat M\to M$ be the blow-up of $N$. Then the induced diffeomorphism $\hat f\colon \hat M\to\hat M$ (flow $\hat\varphi^t\colon \hat M\to\hat M$) is partially hyperbolic.
\end{theoremM}
The same result remains true if we assume that $E^s$ and $E^u$ are smooth distributions rather than assuming their joint integrability to the $N$-fibers~(\ref{eq_fiberwise2}). We do not pursue the proof of such modification here because all the examples which we consider here do satisfy~(\ref{eq_fiberwise2}).
Also, we would like to remark that the Main Theorem generalizes in a fairly straightforward way to the case when the fiber diffeomorphism $f_N\colon N\to N$ is assumed to be partially hyperbolic rather than Anosov.

\begin{remark}
If $f$ preserves a volume $vol$ then diffeomorphism $\hat f\colon \hat M\to \hat M$ preserves a smooth measure $\pi^*vol$ whose density vanishes on the exceptional set. It would be very interesting to obtain a volume preserving version of the Main Theorem. However, it doesn't seem that this can be done in a straightforward way. One can apply the trick of Katok-Lewis~\cite{KL}, which is to alter the smooth structure at $N$, and obtain a volume preserving induced diffeomorphism $\tilde f\colon \hat M\to \hat M$. Then it becomes clear that, in order to retain partial hyperbolicity, stronger domination property of $N$ is needed. This would make impossible many of examples which we construct in this paper. On top of this, controlling the center distribution (estimates in Section~\ref{sec_center}) becomes a very formidable problem.
\end{remark}

\begin{example}
We demonstrate that the Main Theorem provides new examples. Let $H$ be the 3-dimensional Heisenberg group of upper-triangular $3\times 3$ matrices. There exists a lattice $\Gamma\subset H\times H$ and a hyperbolic automorphism $H\times H\to H\times H$ such that $M\stackrel{\mathrm{def}}{=} H\times H/\Gamma$ is a compact nilmanifold and the automorphism induces an Anosov diffeomorphism $A\colon M\to M$. Construction of such Anosov diffeomorpisms is due to Smale-Borel~\cite{smale}. It is clear from the construction that $A$ can be viewed as a partially hyperbolic diffeomorphism with a 4-dimensional center distribution. When considered this way $A$ has  an Anosov torus $\T^2\subset M$ and, after making a perturbation in a neighborhood of this torus, the Main Theorem applies and yields a partially hyperbolic diffeomorphism $\hat A\colon\hat M\to \hat M$. Of course the new diffeomorphism is not Anosov anymore and has fixed points of indices 1 and 5. One can check that the manifold $\hat M$ (unlike $M$) is rich in higher homotopy groups (the universal cover of $\hat M$ is homotopy equivalent to the infinite wedge sum $\bigvee_i\S^4_i$) and one can deduce, by looking at $\pi_4$, that the universal cover of $\hat M$ is not diffeomorphic to any Lie group. Also note that $\hat A$ cannot be homotopic to a time one map of a geodesic flow simply because $\hat M$ is even dimensional. We discuss the construction of $\hat A$ in more detail later, see Example~\ref{nil_example}.
\end{example}


\subsection{The structure of the paper}
In the next section we present some variations of the Main Theorem such as the complex blow-up version and the connected sum construction for partially hyperbolic diffeomorphisms. Section 4 is devoted to discussion of examples to which our results apply, both diffeomorphism and flow examples. The last Section 5 contains the proofs.

The author would like to thank Federico Rodriguez Hertz for many useful conversations and feedback on the first draft of this paper.

\section{Some variations of the Main Theorem}
\subsection{A complex blow-up}
We describe a version of the Main Theorem where one uses a complex blow-up instead of a real one. This amounts to a different surgery on the neighborhood of $N$ which does not affect the fundamental group of the manifold.

As before, we assume that $N\subset M$ is a dominant Anosov submanifold for a partially hyperbolic diffeomorphism $f\colon M\to M$. Further we assume that $N$ has even codimension $2k$ and that the neighborhood of $N$ is identified with $\D^k_\C\times N$, where $\D^k_\C=\{x\in\C^k:\,\,\|x\|<1\}$ so that $f$ is locally fiberwise on $\D^k_\C\times N$, that is, stable and unstable distributions satisfy~(\ref{eq_fiberwise2}) and the restriction $f|_{\D^k_\C\times N}$ is given by
$$
(x,y)\mapsto (Ax, f_N(y)),
$$
where $A$ is a hyperbolic complex-linear map.

With such a setup we can follow through the discussion of Section~\ref{sec_setup} simply by working over $\C$ instead of $\R$, and arrive at the induced map $\hat f\colon\hat M_\C\to\hat M_\C$, where $\hat M_\C$ is obtained from $M$ by replacing $\D^k_\C\times N$ with $\tilde \D^k_\C\times N$. Here $\tilde \D^k_\C$ is the complex blow-up of $\D^k_\C$ and one can check that $\tilde\D^k_\C$ is diffeomorphic to $\D^k_\C\# \overline{\C P}^k$, see \eg~\cite[Proposition 2.5.8]{Huy}. The setup of the complex blow-up for flows is analogous.

\begin{theorem} 
\label{thm_complex}
Let $f\colon M\to M$ ($\varphi^t\colon M\to M$) be a partially hyperbolic diffeomorphism (flow) and let $N\subset M$ be an invariant submanifold which satisfies the above assumptions. Then the induced partially hyperbolic diffeomorphism $\hat f\colon \hat M_\C\to\hat M_\C$ (flow $\hat\varphi^t\colon\hat M_\C\to\hat M_\C$) is partially hyperbolic.
\end{theorem}
The proof of this theorem is similar to the proof of the Main Theorem and we discuss necessary modifications in Section~\ref{sec_complex}

\subsection{Surgery variations}
First we remark that the submanifold $N$ does not have to be connected. For example $N$ could have several connected components which are being cyclically permuted by $f$.
\subsubsection{Multiple blow-ups} Another observation is that the blow-up procedure could be carried out with respect to several Anosov submanifolds. For example, assume that $N_1, N_2\subset M$ are both Anosov submanifolds such that the Main Theorem applies to $N_1$ and Theorem~\ref{thm_complex} applies to $N_2$. Then, after performing the real blow-up of $N_1$ we obtain a partially hyperbolic diffeomorphism $\hat f\colon\hat M\to\hat M$ which still leaves $N_2$ invariant. Because the blow-down map $\pi\colon\hat M\to M$ preserves all dynamical structures (including the stable and unstable distributions) away from the exceptional set, we can further perform a complex blow-up at $N_2\subset \hat M$ and obtain a partially hyperbolic diffeomorphism $\doublehat f\colon\doublehat M\to\doublehat M$ .

Same remark is applicable in the flow case.
\subsubsection{Connected sums along the invariant submanifolds} 
\label{section_cs}
Now assume that $N_i\subset M_i$ are invariant under $f_i\colon M_i\to M_i$, $i=1,2$, and that both $f_1$ and $f_2$ satisfy the assumptions of the Main Theorem. Moreover assume that both $N_1$ and $N_2$ are diffeomorphic to a manifold $N$ and that the local forms of $f_1$ and $f_2$ at the invariant submanifold are the same (after identifying both neighborhoods of $N_1$ and $N_2$ with $\D^k\times N$)
$$
(x,y)\mapsto (Ax,f_N(x)).
$$ 
Then one can glue $f_1$ and $f_2$ together as follows. First, perform the ``spherical" blow-up for both $N_1$ and $N_2$; that is, we replace $\D^k\times N_i$ with $\bar \D^k\times N_i$, $i=1,2$. Here $\bar \D^k$ is defined as
\begin{equation*}
\bar \D^k =\{(x, r(x)):\, x\in\D^k, x\in r(x)\},
\end{equation*}
where $r(x)$ is the ray based at $0$ and passing through $x$. Both resulting manifolds $\bar M_1$ and $\bar M_2$ have boundaries diffeomorphic to $\S^{k-1}\times N$. Each $f_i$ induces a diffeomorphism $\bar f_i\colon\bar M_i\to\bar M_i$, $i=1,2$. Moreover, on the neighborhood of the boundary $\S^{k-1}\times N\times [0,1)$, both $\bar f_1$ and $\bar f_2$ have the same form
$$
(s,y,t)\mapsto (\hat As, f_N(y), a(s)t),
$$
where $\hat A\colon\S^{k-1}\to\S^{k-1}$ is the spherical projectivization of $A$ and $a(s)=\|As\|$ (here $s\in\S^{k-1}$ is viewed as a unit vector in $\R^k$). 

Hence we can paste $\bar M_1$ and $\bar M_2$ together to form the connected sum $\bar M$ along $\S^{k-1}\times N$ and also paste $\bar f_1$ and $\bar f_2$ together to form the {\it connected sum} $\bar f\colon \bar M\to \bar M$. The above local form near the boundary implies that $\bar f$ is a smooth diffeomorphism and it easily follows from (the proof of) the Main Theorem that $\bar f$ is partially hyperbolic.

Notice that if $M_1=M_2$ and $f_1=f_2$ then $\bar M$ is the topological double of $M_1$ and $\bar f\colon\bar M\to\bar M$ is a ``partially hyperbolic double" of $f$.
Also notice that if $f\colon M\to M$ admits two different invariant submanifolds $N_i\subset M$, $i=1,2$ then in the same way one can ``spherically" blow-up $f$ at both $N_1$ and $N_2$ and then ``connect sum with itself." 

Finally we notice that the above observations can be combined, such as doing multiple blow-ups and multiple gluings at the same time.

\section{Examples}
This section is devoted to constructions of examples to which the Main Theorem and its variations can be applied. We first discuss discrete time examples and then continuous time examples. All examples considered here are fiberwise Anosov diffeomorphisms or flows. 

\subsection{Fiberwise Anosov diffeomorphisms and flows}
Let $N$ and $X$ be smooth compact manifolds and let $p\colon M\to X$ be a smooth fiber bundle with fiber $F$; this means that $p\colon M\to X$ is a locally trivial fiber bundle given by smooth charts $p^{-1}(\U_\alpha)\simeq \U_\alpha\times F$, $\U_\alpha\subset X$. Given $x\in X$ we denote by $N_x$ the fiber $p^{-1}(x)$. Let $T^{\parallel}M$ be the submanifold of the tangent bundle $TM$ which consists of all vectors tangent to the fibers of $p$
$$
T^{\parallel}M=\bigcup_{x\in X} TN_x
$$

Now, given a smooth fiber bundle $N\to M\to X$ we define fiberwise Anosov systems as follows.  A diffeomorphism $F\colon M\to M$ is called {\it fiberwise Anosov} if there exists a diffeomorphism $f\colon X\to X$, an invariant splitting $T^\parallel M=E^s\oplus E^u$, a constant $\lambda\in(0,1)$ and a smooth Riemannian metric on $T^\parallel M$ such that
\begin{itemize} 
\item
$F$ fibers over $f$; that is, the diagram 
$$
\xymatrix{
M\ar_p[d]\ar^{F}[r] & M\ar_p[d]\\
X\ar^{f}[r] & X
}
$$
commutes; 
\item the following inequalities hold for all unit vectors $v^s\in E^s$ and $v^u\in E^u$
$$
 \|DF(v^s) \| <\lambda< \lambda^{-1}< \|DF(v^u) \|.
$$
\end{itemize}

Similarly, a flow $\Phi^t\colon M\to M$ is called {\it fiberwise Anosov} if there exists a flow $\varphi^t\colon X\to X$, an invariant splitting $T^\parallel M=E^s\oplus E^u$, a constant $\lambda\in(0,1)$ and a smooth Riemannian metric on $T^\parallel M$ such that
\begin{itemize} 
\item
$\Phi^t$ fibers over $\varphi^t$; that is, the diagram 
$$
\xymatrix{
M\ar_p[d]\ar^{\Phi^t}[r] & M\ar_p[d]\\
X\ar^{\varphi^t}[r] & X
}
$$
commutes for all $t$; 
\item the following inequalities hold for all unit vectors $v^s\in E^s$, $v^u\in E^u$ and $t\ge 1$
$$
 \|D\Phi^t(v^s) \| <\lambda^t< \lambda^{-t}< \|D\Phi^t(v^u) \|.
$$
\end{itemize}

\subsection{Examples of fiberwise Anosov dynamical systems}
In order to present examples to which the Main Theorem can be applied we will consider smooth fiber bundles with torus fiber and fiberwise Anosov diffeomorphisms and flows whose fiberwise dynamics is affine. 

\subsubsection{Principal fiber bundles and $B$-diffeomorphisms}
Recall that a smooth fiber bundle $\T^d\to M\to X$ is called {\it principal} if $M$ admits a smooth free $\T^d$-action $y\mapsto y+ g$, $g\in \T^d$, whose orbits are precisely the fibers of the bundle. Hence, all torus fibers of a principal torus bundle are canonically identified with $\T^d$ up to a translation. Given an automorphism $B\colon\T^d\to\T^d$ we say that a diffeomorphism $F\colon M\to M$ is a $B$-diffeomorphism if
$F(y_0+ y)=F(y_0)\cdot B(y)$
for all $y_0\in M$ and all $y\in\T^d$. In other words, $F$ preserves the fibers and is locally given by the formula
\begin{equation}
\label{eq_Bmap}
F\colon (x,y)\mapsto (f(x), By+\varphi(x)),\,\,\, (x,y)\in \U_\alpha\times\T^d,
\end{equation}
where $\varphi\colon \U_\alpha\to\T^d$ depends on the choice of charts at $x$ and at $f(x)$. Clearly, if $B$ is hyperbolic then a $B$-diffeomorphism is fiberwise Anosov. We refer to~\cite{GORH} for a thorough discussion of $B$-diffeomorphisms.

Potentially, $B$-diffeomorphisms with hyperbolic (or partially hyperbolic) $B\in SL(d,\Z)$ provide a rich class of partially hyperbolic diffeomorphisms. Theorem~8.2 in~\cite{GORH} gives a general criterion for partial hyperbolicity of a $B$-diffeomorphism. Loosely speaking, it says that a $B$-diffeomorphism is partially hyperbolic provided that the base dynamics is dominated by $B$. However, as explained in~\cite{GORH}, it is difficult to create partially hyperbolic $B$-diffeomorphism of non-trivial fiber bundles as there is no known general method of verifying the assumption of this criterion, \ie controlling the base dynamics of the $B$-diffeomorphisms. 

One application of The Main Theorem is that it provides a surgery machinery to create new partially hyperbolic $B$-diffeomorhisms from the known examples. This is achieved by applying the Main Theorem and Theorem~\ref{thm_complex} to invariant and periodic torus fibers  and by taking connected sums and ``self-connected sums" along invariant torus fibers. We proceed with description of examples.

\subsubsection{Examples of partially hyperbolic $B$-diffeomorphisms} Here we present some known examples of $B$-diffeomorphisms which are partially hyperbolic: products, the example of~\cite{GORH}, nilmanifold automorphisms; and explain how our results can be applied to these examples.
\begin{example}[Product]
\label{ex_product} 
The trivial example of a $B$-diffeomorphism is, of course, the product diffeomorphism ${id_X\times B\colon X\times\T^d\to X\times\T^d}$, where $B$ is hyperbolic. 
Note that, formally speaking, the Main Theorem does not apply to this example because we do not have a hyperbolic fixed point in the base, however we can modify it so that the Main Theorem becomes applicable. Namely, let $A\colon\R^k\to\R^k$ be a hyperbolic linear automorphism, which is dominated by $B$; \ie\footnote{This is simply a restatement of the domination assumption~(\ref{eq_dominant})}
\begin{equation}
\label{eq_dominant3}
\begin{split}
&\frac{\max\{|\lambda|:\, \lambda\in spec(A)\}}{\min\{|\lambda|:\, \lambda\in spec(A)\}}<\min\{|\lambda|: \lambda\in spec(B),|\lambda|>1\}\\
&\frac{\min\{|\lambda|:\, \lambda\in spec(A)\}}{\max\{|\lambda|:\, \lambda\in spec(A)\}}>\max\{|\lambda|: \lambda\in spec(B),|\lambda|<1\}
\end{split}
\end{equation}
Then one can homotope $id_X$ to a diffeomorphism $f\colon X\to X$ so that $f$ coincides with $A$ on a disk $\D^k\subset X$ and $f\times B$ is still partially hyperbolic. Then the Main Theorem applies and yields a partially hyperbolic diffeomorphism $\widehat{f\times B}\colon (X\#\R P^k)\times\T^d\to (X\#\R P^k)\times\T^d$. This is not of much interest as this diffeomorphism is merely a product again. However, the diffeomorphism $f\times B$ becomes much more meaningful for connected sum constructions (which we explain once we have more examples to connect sum with).
\end{example}

\begin{example}[Over the $K3$-surface] 
 \label{GORH_example}
Given a hyperbolic automorphism $A\colon\T^2\to\T^2$ where exists a principal fiber bundle $\T^2\to M\to K3$ over the $K3$-surface  whose total space $M$ is simply connected and a partially hyperbolic $A^2$-map $F\colon M\to M$ which fibers over $f\colon K3\to K3$ (see~\cite{GORH}).
Further, it is easy to see from the construction in~\cite{GORH} that (after passing to a finite iterate) the base map $f\colon K3\to K3$ has a fixed point $x_0$ such that on a disk $\D^4$ centered at $x_0$ the base diffeomorphism $f$ is given by $x\mapsto A\oplus A(x)$. Note that $A^2$ does not dominate $A\oplus A$ as we require strict inequalities in~(\ref{eq_dominant3}). However we can perturb $f$ in $C^1$ topology, and $F$ accordingly, so that $F$ is still partially hyperbolic and $f$ about $x_0$ is given by $x\mapsto A'\oplus A'(x)$, where $A'$ has eigenvalues closer to 1 and hence is dominated by $A^2$. Then locally, in the neighborhood $\D^4\times\T^2_{x_0}$ of the invariant fiber $\T^2_{x_0}=p^{-1}(x_0)$, diffeomorphism $F$ is given by
$$
(x,y)\mapsto (A'\oplus A'(x), A^2y+\varphi(x)).
$$
In order to apply the Main Theorem at $\T_{x_0}$ we need to further modify $F$ in order to bring it locally fiberwise form~(\ref{eq_fiberwise}). Namely, we replace $F$ with a diffeomorphism $F'$ which coincides with $F$ outside $\D^4\times\T^2_{x_0}$ and is given by
$$
(x,y)\mapsto (A'\oplus A'(x), A^2y+\psi(x))
$$
on $\D^4\times\T^2_{x_0}$, where $\psi$ coincides with $\varphi$ near the boundary $\partial\D^4$ and equals to 0 on smaller disk so that on the smaller neighborhood $F'$ has the locally fiberwise form
$$
(x,y)\mapsto (A'\oplus A'(x), A^2y).
$$
Because this procedure does not affect the base map $f$, the diffeomorphism $F'$ is still partially hyperbolic by~\cite[Theorem 8.2]{GORH}. Now both the Main Theorem and Theorem~\ref{thm_complex} could be applied at $x_0$ and yield partially hyperbolic diffeomorphisms 
$\hat {F'}\colon \hat M\to \hat M$ and $\hat{F'_{\C}}\colon \hat M_\C\to\hat M_\C$.
\end{example}

\begin{example}[Nilmanifold automorphisms]
 \label{nil_example}
 Recall that by work of Mal$'$cev~\cite{Mal} any compact nilmanifold $M$ can be represented as a homogeneous coset space
 $$
 M=N/\Gamma,
 $$
 where $N$ is a simply connected nilpotent Lie group and $\Gamma\subset N$ is a cocompact lattice. Further, again by~\cite{Mal}, if $\ZZ(N)\simeq\R^d$ is the center of $N$, then $\Gamma\cap\ZZ(N)$ is a cocompact lattice in $\ZZ(N)$ and, hence, $\ZZ(N)/\Gamma\cap\ZZ(N)$ can be identified with the torus $\T^d$. Note that $\ZZ(N)$ act on $N$ by left translation and this action descends to a free action of $\T^d$ on $M$. Hence $M$ is the total space of a principal fiber bundle
 $$
 \T^d\to M\to X,
 $$
 where $X=N/\Gamma\cdot\ZZ(N)$ is nilmanifold modeled on the simply connected nilpotent Lie group $N/\ZZ(N)$. This bundle is non-trivial provided that $N$ is non-abelian.
\end{example}
Now let $C\colon N\to N$ be an automorphism and let $B$ be its restriction to the characteristic subgroup $\ZZ(N)$. Assume that $C$ preserves a cocompact lattice $\Gamma$ and that $B$ is hyperbolic. Then $C$ induces a nilmanifold automorphism $C\colon M\to M$ and $B$ becomes a hyperbolic toral automorphism. Further we can view $C$  as $B$-diffeomorphism over the quotient automorphism of $X$.

Some nilmanifold automorphisms of this type can be perturbed to $B$-diffeomorphisms to which the Main Theorem applies. For instance, such examples can be found within the classical Borel-Smale family of Anosov automorphism of a 6-dimensional 2-step nilmanfold $M$ (see original description~\cite{smale} and~\cite{BW} for a thorough exposition). Namely given a hyperbolic automorphism $A\colon\T^2\to\T^2$ there exists an automorphism $F\colon M\to M$ which fibers over $A\oplus A$
$$
\xymatrix{
M\ar[d]\ar^{F}[r] & M\ar[d]\\
X\ar^{A\oplus A}[r] & X
}
$$
Now given an invariant fiber $\T^2_{x_0}$ one can perform exactly the same modifications in the neighborhood of $\T^2_{x_0}$ to obtain a partially hyperbolic $A^2$-diffeomorphism to which the Main Theorem and its modifications apply.

\subsubsection{Further surgery examples of partially hyperbolic $B$-diffeomorphisms}
We would like to point out that connect-summing along invariant tori explained in Subsection~\ref{section_cs} works well for all of the above examples. Indeed, the local form of base map $A'\oplus A'$ near the fixed point is the same for the latter examples and we can also choose the same local form for the the product example. Further, by arranging for multiple invariant fibers, a repeated connected sum can be taken which yield a partially hyperbolic $A^2$-diffeomorphisms of non-trivial principal $\T^2$-bundles over manifolds of the form
$M\#nK3\#m\T^4$, where $M$ is an arbitrary manifold coming from the product example.

Finally we notice, that ``self-connected sum" construction also applies to these examples with two or more invariant fibers.

\subsubsection{Examples of fiberwise Anosov flows} 
Here we describe examples of fiberwise Anosov flows on torus bundles $\T^d\to M\to X$ whose structure group is $SL(d,\Z)$ to which the flow version of the Main Theorem applies.

\begin{example}[Suspension]
 \label{ex_suspension}
 Consider a product $f\times B\colon M\times\T^d\to M\times\T^d$, where $B\colon\T^d\to\T^d$ is an automorphism. Let $(M\times\T^d)_{f\times B}$ be the mapping torus of $f\times B$, \ie
 $$
 (M\times\T^d)_{f\times B}=M\times\T^d\times[0,1]/(x,y,1)\sim (f(x),By,0).
 $$
 We view $(M\times\T^d)_{f\times B}$ as the total space of the torus bundle over the mapping torus $M_f$ of $f\colon M\to M$.
 Then the suspension flow $\Phi^t\colon  (M\times\T^d)_{f\times B}\to  (M\times\T^d)_{f\times B}$ fibers over the suspension flow $\varphi^t\colon M_f\to M_f$ of $f\colon M\to M$
 $$
 \xymatrix{
(M\times\T^d)_{f\times B}\ar[d]\ar^{\Phi^t}[r] & (M\times\T^d)_{f\times B}\ar[d]\\
M_f\ar^{\varphi^t}[r] & M_f
}
 $$
 Moreover, if $B$ is hyperbolic then $\Phi^t$ is fiberwise Anosov.
 
 Further assume that $f$ has a hyperbolic fixed point $p$ and is given by $x\mapsto Ax$ in a chart centered at $p$. Then the restriction of $\varphi^t$ to the orbit of $p$ is the unit-speed flow on the circle $S^1$ and the restriction of the fiberwise Anosov flow is the suspension flow of $B$; that is
 $$
  \xymatrix{
\T^d_ B\ar[d]\ar^{\Phi_B^t}[r] & \T^d_B\ar[d]\\
S^1\ar^{\varphi^t}[r] & S^1
}
 $$
It is easy to see that, after choosing appropriate coordinates in the neighborhood of the invariant submanifold $\T^d_B\subset (M\times\T^d)_{f\times B}$ the flow $\Phi^t$ is given by
$$
(x,y)\mapsto (A^tx,\Phi_B^t(y)),
$$
where $A^t$ is the hyperbolic saddle whose time-1 map is $A$ and $\Phi_B^t$ is the suspension flow on $\T^d_B$. Now we assume that $B$ dominates $A$ as in~(\ref{eq_dominant3}) (for example one can pick $f$ first and then pick $B$ so that~(\ref{eq_dominant3}) holds). Then the Main Theorem applies to $\Phi^t$ and yields a fiberwise Anosov flow $\hat{\Phi}^t\colon\widehat{(M\times\T^d)}_{f\times B}\to \widehat{(M\times\T^d)}_{f\times B}$ by blowing-up the mapping torus $\T^d_B$. One can check that the resulting flow is the suspension of $\hat f\times B\colon\hat M\times\T^d\to\hat M\times \T^d$, where $\hat f$ is the blow-up of $f$ at $p$. However, this example still has value as a building block for connected sum constructions.

Of course, more generally, one can use any of the discrete time fiberwise Anosov examples discussed before in place of $B$ in the suspension construction.
\end{example}

\begin{example}[Higher rank suspension] 
 \label{ex_suspension2}
 Another way to construct examples, which allows to dispose of taking the product with $f$, is to consider higher rank $k\ge 3$ suspensions. 
 
 Let $B_1, B_2, \ldots B_k$ be commuting automorphisms of the torus $\T^d$. They define an action $\bar B\colon\Z^k\to Aut(\T^d)$. Let $\Z^k$ act on $\T^d\times\R^k$ by
 $$
 \bar n(x,v)=(\bar B(\bar n)x, v-\bar n)
 $$
 The higher rank mapping torus
 $$
 \T^d_{\bar B}=(\T^d\times \R^k)/\Z^k
 $$
 is a smooth closed manifold and the action of $\R^k$ on $\T^d\times\R^k$ by translations
 $$
 u(x,t)=(x, t+u)
 $$
 descends to an $\R^k$ action $F\colon\R^k\times\T^d_{\bar B}\to \T^d_{\bar B}$. This actions fibers over the action of $\R^k$ on $\T^k$
  $$
  \xymatrix{
\T^d_ {\bar B}\ar[d]\ar^{F(u)}[r] & \T^d_{\bar B}\ar[d]\\
\T^k\ar^{u}[r] & \T^k
}
 $$
 
Now given a non-zero primitive integral vector $\bar n\in\Z^k$ we obtain the flow $\Phi^t_{\bar n}\colon \T^d_{\bar B}\to \T^d_{\bar B}$ by taking the restriction
$
\Phi^t_{\bar n}=F(t\bar n)$. Flow $\Phi^t_{\bar n}$ fibers over a periodic flow on $\T^k$. Assume that $\bar B(\bar n)$ is hyperbolic (which is easy to arrange) and pick an $\Phi^t_{\bar n}$-invariant mapping torus $N\stackrel{\mathrm{def}}{=}\T^d_{\bar B(\bar n)}\subset \T^d_ {\bar B} $ over a periodic orbit in the base. The normal neighborhood of $N$ can be identified with $\D^{k-1}\times N$ and, locally, the flow is given by
$$
(x,y)\to(x,\Phi^t_N(y)),
$$
where $\Phi^t_N$ is the suspension flow of $\bar B(\bar n)$. We can perturb the flow locally so that the local form becomes
$$
(x,y)\to(A^tx,\Phi^t_N(y)),
$$
where $A^t$ is a ``slow" hyperbolic saddle. Now the Main Theorem applies to the Anosov submanifold $N$ and yields a partially hyperbolic flow on $\hat \T^d_ {\bar B}$. Further, one can form a connected sum of this example with the previous Example~\ref{ex_suspension}.
\end{example}

\begin{example}[Tomter example: suspension of the geodesic flow]
Let $G=PSL(2,\R)$ and let $\Gamma\subset G$ be a torsion-free cocompact lattice acting on $G$ by right multiplication. The the geodesic flow $d^t$ on the unit tangent bundle of a closed surface $T^1S=G/\Gamma$ is given by left multiplication by $diag(e^{t/2}, e^{-t/2})$. Let $\rho\colon\Gamma\to GL(4,\Z)$ be a representation. Then the semi-direct product $\Gamma_\rho\!\!\ltimes \Z^4$ acts on the right on $G\times\R^4$ as follows:
$$
(\gamma,\bar n)\colon (g,v)\mapsto (g\gamma, \rho(\gamma^{-1})v+\bar n)
$$
This action is  smooth, free, properly discontinuous and cocompact. Therefore the quotient $M\stackrel{\mathrm{def}}{=}G\times\R^4/\Gamma_\rho\!\!\ltimes\Z^4$ is a closed smooth manifold and it is easy to see that $M$ is the total space of the fiber bundle $\T^4\to M\to T^1S$ whose structure group is $Im(\rho)\subset GL(4,\Z)$. Clearly the action of $\Gamma_\rho\!\!\ltimes \Z^4$ fibers over the action of $\Gamma$ on $G$ and the product flow $d^t\times id\colon (g,v)\mapsto(d^tg,v)$ descends to a flow $\Phi^t\colon M\to M$ which fibers over the geodesic flow:
$$
  \xymatrix{
M\ar[d]\ar^{\Phi^t}[r] & M\ar[d]\\
T^1S\ar^{d^t}[r] & T^1S
}
$$
Tomter~\cite[Chapter 4]{tomter} proved that one can arrange representation $\rho$ so that this flow is fiberwise Anosov (and, in fact, a homogeneous Anosov flow).

Further we assume that the fiberwise hyperbolicity dominates the geodesic flow in the base so that $\Phi^t$ is a partially hyperbolic with center distribution being transverse to the torus fibers. Let $\alpha$ be a closed geodesic in $T^1S$ and let $B\colon \T^4\to\T^4$ be the monodromy automorphism over $\alpha$. Then the mapping torus $\T^4_B$ over $\alpha$ is a $\Phi^t$ invariant Anosov submanifold and a calculation shows that in a neighborhood $\D^2\times \T^4_B$ the flow $\Phi^t$ is given by
$$
(x_1, x_2, y)\mapsto (e^tx_1, e^{-t}x_2, \Phi^t_By),
$$
where $\Phi^t_B$ is the suspension flow on $\T^4_B$. Now assume that the length $T$ of $\alpha$ is sufficiently small so that $B$ dominates $diag(e^T, e^{-T})$. Under this assumption the Main Theorem applies to $\T^4_B$ and yields a partially hyperbolic fiberwise Anosov flow $\hat \Phi^t\colon\hat M\to\hat M$ over the blow-up of the geodesic flow along $\alpha$.
\end{example}

\begin{remark} 
We point out that verifying the above assumptions --- partial hyperbolicity and existence of a short geodesic --- is a non-trivial matter. The difficulty comes from the fact that Tomter's approach is to work with an arithmetic lattice $\Gamma\subset G$ so that $\rho\colon \Gamma\to GL(4,\Z)$ extends to a representation $\rho\colon G\to GL(4,\R)$. Then $M$ is the homogeneous space $G_\rho\!\!\ltimes \R^4/ \Gamma_\rho\!\!\ltimes\Z^4$ and $\Phi^t$ is a homogeneous flow for which Tomter is able to verify the fiberwise Anosov property. The author plans a separate paper on fiberwise Anosov dynamical systems where the Tomter example will be revisited and the above assumptions verified. The author also plans to describe further fiberwise Anosov flows which are not homogeneous and to which the Main Theorem can be applied (Note that the above examples are homogeneous.)
\end{remark}

\begin{remark} Note also that according to our definitions the time-1 map of a fiberwise Anosov flow is a fiberwise Anosov diffeomorphism. Hence the discrete time version of the Main Theorem applies to the time-1 maps of the above examples. Also one can form partially hyperbolic connected sums of these time one maps with the product Example~\ref{ex_product}.
\end{remark}

\section{The proof of the Main Theorem}
\subsection{A family of Riemannian metrics on $\tilde \D^k$}
Let $\e_0$ be a small positive constant. We begin the proof with a description of a family of Riemannian metrics $g_\e$, $\e\in(0,\e_0)$, on $\tilde \D^k$~(\ref{eq_D}). These metrics will be constructed so that each metric $g_\e$ in the family coincides with the canonical flat metric near the boundary of $\tilde \D^k$ and restricts to the round metric of curvature $\e^{-2}$ on $\R P^{k-1}\subset\tilde \D^k$.

First we give an alternate description of $\tilde \D^k$ given by~(\ref{eq_D}) as a quotient manifold. Consider
\begin{equation*}
\bar \D^k =\{(x, r(x)):\, x\in\D^k, x\in r(x)\},
\end{equation*}
where $r(x)$ is the ray based at $0$ and passing through $x$. Polar coordinates on $\D^k$ yield the identification $\bar\D^k\simeq \S^{k-1}\times [0,1)$. Under this identification the map $\bar\D^k\to\tilde\D^k$, which sends the ray to the unique line containing the ray, becomes the quotient map
$$
\S^{k-1}\times [0,1)\to \S^{k-1}\times [0,1)/\sim
$$
with the gluing $\sim$ is given by $(s,0)\sim(-s,0)$, where $s\mapsto-s$ is the antipodal map.

Let $\rho\colon [0,1)\to \R^+$ be a smooth function which is $C^\infty$ flat at $0$ and let $ds^2$ be the standard round metric of curvature 1 on $\S^{k-1}$. Then the warped metric 
\begin{equation*}
\label{eq_warped}
dt^2+\rho(t)^2ds^2
\end{equation*}
 (see \eg~\cite[Chapter 1]{Petersen}) on $\bar\D^k$ factors through the quotient map to a smooth Riemannian metric on $\tilde \D^k$.  Hence we can define the family of metrics $g_\e$, $\e\in(0,\e_0)$, on $\tilde\D^k$ in the warped form
 $$
 g_\e=dt^2+\rho_\e(t)^2ds^2,
 $$
 where $\rho_\e\colon[0,1)\to\R^+$ is chosen so that
 $$
 \rho_\e(t)=
 \begin{cases} \e, \,\,\,t\le\e/2\\
 t, \,\,\,t\ge\e
 \end{cases}
 $$
 and 
 \begin{equation}
 \label{eq_rho_condition}
 t\le \rho_\e(t)\le\e,\,\,\,\,\, t\in[\e/2,\e].
 \end{equation}
 
 Let $can$ be the canonical Euclidean metric on $\D^k$, $can=dx_1^2+dx_2^2+\ldots dx_k^2$. In the polar coordinates $(t,s)\in[0,1)\times\S^{k-1}$ this metric takes warped form $can=dt^2+t^2ds^2$. Hence, by the definition of $g_\e$, the blow-down map $\pi\colon(\tilde \D^k, g_\e)\to (\D^k, can)$ is an isometry when restricted to $\{(t,s): t>\e\}$. Also note that the restriction of $g_\e$ to $\{(t,s): t<\e/2\}$  is the direct sum $dt^2+\e^2ds^2$.
 
\subsection{Basic domination estimate} Here will prove a basic lemma which is the core for the proof of partial hyperbolicity of $\hat f\colon \hat M\to\hat M$. 

Recall that $A\colon \D^k\to\D^k$ is a hyperbolic linear automorphism and $\tilde A\colon\DD^k\to\DD^k$ is the induced diffeomorphism introduced in Section~\ref{sec_setup}.\footnote{More precisely, $A\colon\R^k\to\R^k$ is a hyperbolic linear automorphism and we abuse notation by writing $A\colon\D^k\to\D^k$ for the restriction $A|_{A^{-1}(\D^k)}$. Such abuse of notation is harmless because we are only interested in local dynamics.}

Because $E^c$ is ``horizontal" on $N$ the domination assumption~(\ref{eq_dominant}) implies that
$$
\lambda'\le\min\{|\lambda|,\,\,\lambda\in spec(A)\},\,\,\,\,\,\, \mu'\ge \max\{|\lambda|,\,\,\lambda\in spec(A)\}.
$$
Hence, again by~(\ref{eq_dominant}) (note the strict inequality), there exists $\xi>0$ such that if we let
\begin{equation*}
\nu=\max\{|\lambda|,\,\,\lambda\in spec(A)\}+\xi,\,\,\,\tau=\min\{|\lambda|,\,\,\lambda\in spec(A)\}-\xi.
\end{equation*}
then, by the second inequality of~(\ref{eq_dominant})
\begin{equation}
\label{eq_nu_tau}
\lambda<\frac{\tau}{\nu}<\frac{\nu}{\tau}<\mu.
\end{equation}
Denote by $\|\cdot\|_\e$ the norm induced by $g_\e$.
\begin{lemma} 
\label{lemma_estimate}
Given the the induced map $\tilde A$ and the family of metrics $g_\e$ as above, there exists a constant $C>0$ (independent of $\e$) such that for any finite orbit $\{x, \tilde Ax, \tilde A^2x,\ldots \tilde A^{n}x\}\subset\DD^k$ and any $v\in T_x\DD^k$ the following inequalities hold
$$
C^{-1}(\tau/\nu)^n\|v\|_\e\le\|D\tilde A^nv\|_\e\le C (\nu/\tau)^n\|v\|_\e
$$
\end{lemma}
For the proof of the lemma recall that $(\DD^k, g_\e)$ is partitioned into three subdomains \footnote{We will continue using subscript decorations to represent subdomains with various restrictions on the radial coordinate.}

\begin{equation*}
\begin{split}
&\DD^k_{>\e}=\{(t,s)\in\DD^k: t>\e\},\\
&\DD^k_{[\e/2,\e]}=\{(t,s)\in\DD^k: t\in[\e/2,\e]\},\\
&\DD^k_{<\e/2}=\{(t,s)\in\DD^k: t<\e/2\}.
\end{split}
\end{equation*}
where the first one is flat, the second one is a ``transition" domain, and the last one is metrically a product. Because $A$ is hyperbolic, any finite orbit $\{x, \tilde Ax, \tilde A^2x,\ldots \tilde A^{n}x\}\subset\DD^k$ can be split into five segments (some of which could be empty)
\begin{equation}
\label{eq_orbit}
\{x, \tilde Ax, \tilde A^2x,\ldots \tilde A^{n}x\}=O_1\cup O_2\cup O_3\cup O_4\cup O_5,
\end{equation}
where $O_1\cup O_5\subset \DD^k_{>\e}$, $O_2\cup O_4\subset \DD^k_{[\e/2,\e]}=\{(t,s)\in\DD^k: t\in[\e/2,\e]\}$ and $O_3\subset \DD^k_{<\e/2}=\{(t,s)\in\DD^k: t<\e/2\}$. Using this partition we will reduce the proof of Lemma~\ref{lemma_estimate} to the following special cases.
\begin{lemma} 
\label{lemma_estimate2}
Lemma~\ref{lemma_estimate} holds true if one additionally assumes that $\{x, \tilde Ax, \tilde A^2x,\ldots \tilde A^{n}x\}\subset \DD^k_{>\e}$. In fact, a better estimate holds
$$
C^{-1}\tau^n\|v\|_\e\le\|D\tilde A^nv\|_\e\le C \nu^n\|v\|_\e
$$
\end{lemma}
This statement easily follows from basic linear algebra and the fact that the metric $g_\e$ on $\DD^k_{>\e}$ is the standard Euclidean metric.
\begin{lemma}
\label{lemma_core}
 Lemma~\ref{lemma_estimate} holds true if one additionally assumes that $\{x, \tilde Ax, \tilde A^2x,\ldots \tilde A^{n}x\}\subset \DD^k_{<\e/2}$. 
\end{lemma}
We will prove the above lemma later. Now we proceed with the proof of Lemma~\ref{lemma_estimate} assuming the Lemma~\ref{lemma_core}.
\begin{proof}[Proof of Lemma~\ref{lemma_estimate}]
Denote by $\|\cdot\|$ the flat metric on $\DD^k\backslash \R P^{k-1}$, that is, the pullback $\pi^*(can)$ from $(\D^k\backslash \{0\}, can)$ . Let $y\in \DD^k_{[\e/2,\e]}$ and $v\in T_y \DD^k$. Then bounds~(\ref{eq_rho_condition}) imply that
$$
1\le \frac{\|v\|_\e}{\|v\|}\le 2.
$$
Similar bound holds on a larger domain. Namely, for $y\in \DD^k_{[\e/2,\e]}\cup \tilde A(\DD^k_{[\e/2,\e]})\cup\tilde A^{-1}(\DD^k_{[\e/2,\e]})$ and $v\in T_y \DD^k$
\begin{equation}
\label{eq_bound1}
1\le \frac{\|v\|_\e}{\|v\|}\le K,
\end{equation}
where $K$ depends on $A$ but is independent of $\e$. Indeed, this is easy to see from the fact that $\big(\DD^k_{[\e/2,\e]}\cup \tilde A(\DD^k_{[\e/2,\e]})\cup\tilde A^{-1}(\DD^k_{[\e/2,\e]})\big)\cap \tilde \D^k_{t<c\e}=\varnothing$, for some $c=c(A)<1/2$.

Now we can obtain estimates for the differential $D\tilde A$ as follows. Let $y\in \DD^k_{[\e/2,\e]}\cup\tilde A^{-1}(\DD^k_{[\e/2,\e]})$ and $v\in T_y \DD^k$. Then, using~(\ref{eq_bound1}) and the obvious estimate 
$$
|A^{-1}|^{-1}\le\frac{\|D\tilde Av\|}{\|v\|}\le|A|,
$$
we have
\begin{equation}
\label{eq_bound2}
\begin{split}
& \frac{\|D\tilde Av\|_\e}{\|v\|_\e}=\frac{\|D\tilde Av\|_\e}{\|D\tilde A v\|}\cdot\frac{\|D\tilde Av\|}{\|v\|}\cdot\frac{\|v\|}{\|v\|_\e}\le K|A|   \\
& \frac{\|D\tilde Av\|_\e}{\|v\|_\e}=\frac{\|D\tilde Av\|_\e}{\|D\tilde A v\|}\cdot\frac{\|D\tilde Av\|}{\|v\|}\cdot\frac{\|v\|}{\|v\|_\e}\ge K^{-1}|A^{-1}|^{-1}
\end{split}
\end{equation}
Recall that the finite orbit is decomposed into five segments~(\ref{eq_orbit}). It a standard fact, which follows from dynamics of hyperbolic saddle, that the lengths of $O_2$ and $O_4$ are uniformly bounded by an integer which depends on $A$. Because $A$ commutes with scaling this integer is, in fact, independent of $\e$. 

We can decompose $\|D\tilde A^n v\|/\|v\|$ into the product of five norm ratios according to the splitting~(\ref{eq_orbit}) and notice that the terms which correspond to $O_1$, $O_3$ and $O_5$ are taken care of by Lemmas~\ref{lemma_estimate2} and~\ref{lemma_core}. The terms corresponding to $O_2$ and $O_4$ are uniformly bounded by a constant which is independent of $\e$ because the lengths of these orbit segments are uniformly bounded and uniform estimates~(\ref{eq_bound2}) hold for these orbits segment. Also notice that the transition ratios $\|D\tilde A v\|/\|v\|$, $v\in T_y\tilde \D^k$, when $y\in O_1$ and $f(y)\in O_2$ or $y\in O_2$ and $f(y)\in O_3$ etc., are also taken care of by~(\ref{eq_bound2}). By putting these estimates together we obtain the posited estimate of Lemma~\ref{lemma_estimate} with a constant $C>0$ which is independent of $\e$.
\end{proof}
\begin{proof}[Proof of Lemma~\ref{lemma_core}]
Recall that $(\DD^k_{<\e/2}, g_\e)$ is isometric to $(\S^{k-1}\times[0,\e/2))/\!\!\!\sim\,, \e^2ds^2+dt^2)$. For the purpose of estimating the expansion rate of $\tilde A\colon \DD^k_{<\e/2}\to \DD^k_{<\e/2}$ the identification $\sim$ makes no difference. Hence we can consider the induced map on $(\S^{k-1}\times[0,\e/2), \e^2ds^2+dt^2)$ instead, which we still denote by $\tilde A$. Also note that $(\S^{k-1}\times[0,\e/2), \e^2ds^2+dt^2)$ isometrically embeds into $(\S^{k-1}\times[0,\infty), \e^2ds^2+dt^2)$ and it would be more convenient notation-wise to consider $\tilde A\colon  (\S^{k-1}\times[0,\infty), \e^2ds^2+dt^2)\to (\S^{k-1}\times[0,\infty), \e^2ds^2+dt^2)$. Because $A\colon\R^k\to\R^k$ maps rays to rays, diffeomorphism $\tilde A$ has the skew product form
$$
\tilde A(s,t)=(\hat As, a(s)t),
$$
where $\hat A\colon \S^{k-1}\to\S^{k-1}$ is the projectivization of $A$ and $a\colon\S^{k-1}\to\R^+$ is given by 
$$
a(s)=\frac{\|Av\|}{\|v\|}, v=(s,1).
$$
\begin{claim}
\label{claim_proj}
For any $x\in\S^{k-1}$ and any $v\in T_x\S^{k-1}$ the following estimate holds
$$
C^{-1}(\tau/\nu)^n\|v\|\le\|D\hat A^nv\|\le C (\nu/\tau)^n\|v\|,
$$
where $\|\cdot\|^2=ds^2$.
\end{claim}
Note that this claim is a particular case of Lemma~\ref{lemma_core} for vectors tangent to $\S^{k-1}\times\{0\}\subset \S^{k-1}\times[0,\e)$.

We proceed with the prove of Lemma~\ref{lemma_core} assuming Claim~\ref{claim_proj}. Let
$$
\AA^n(s)=a(s)a(\hat As)a(\hat A^2s)\ldots a(\hat A^{n-1}s).
$$
From definition of $\tau$ and $\nu$ we have that there exists $c_1>0$ such that
$$
\forall n>0,\,\,\,\,c_1^{-1}\tau^n<\frac{\|A^nv\|}{\|v\|}<c_1\nu^n
$$
 which implies
\begin{equation}
\label{eq_AA}
\forall n>0,\,\,\,\,\,c_1^{-1}\tau^n<\AA^n(s)<c_1\nu^n
\end{equation}

Now let $\{(s,t), \tilde A(s,t),\ldots, \tilde A^n(s,t)\}\subset \tilde \D^k_{<\e/2}$ be a finite orbit. Note that 
\begin{equation}
\label{eq_An}
\tilde A^n(s,t)=(\hat A^ns,\AA^n(s)t)
\end{equation}
and, hence, the second coordinate must be less than $\e/2$:
\begin{equation}
\label{eq_eps}
\AA^n(s)t<\e/2.
\end{equation}
By differentiating~(\ref{eq_An}) we obtain the lower diagonal form for the differential
$$
D_{(s,t)}\tilde A^n=
\begin{pmatrix}
D_s\hat A^n & 0\\
t\nabla\AA^n(s) & \AA^n(s)
\end{pmatrix}
$$
We already have estimates on the diagonal entries, but we also need to control the gradient of $\AA^n(s)$. Recall that $\|\cdot\|^2=ds^2$. By taking the gradient of the product we have
\begin{multline*}
\|\nabla\AA^n(s)\|=\left\|\sum_{i=0}^{n-1}\frac{\AA^n(s)}{a(\hat A^i(s))}\nabla(a\circ\hat A^i)(s))\right\|
\\
\le c_2\AA^n(s)\sum_{i=0}^{n-1}\|\nabla(a\circ\hat A^i)(s))\|
\le c_2\AA^n(s)\sum_{i=0}^{n-1}|D\hat A^i_s|\cdot\|\nabla a( \hat A^is)\|\\
\le c_3\AA^n(s)\sum_{i=0}^{n-1}|D\hat A^i_s|
\le c_4 \AA^n(s)\sum_{i=0}^{n-1}\left(\frac\nu\tau\right)^i=c_5\AA^n(s)\left(\frac\nu\tau\right)^n,
\end{multline*}
where for the first inequality we have used the fact that $a$ is uniformly bounded from below, for the third inequality we have used the fact that $\|\nabla a\|$ is bounded and for the forth inequality we have invoked Claim~\ref{claim_proj}.

Let $v=(v_s,v_t)\in T_{(s,t)}\tilde\D^k_{<\e/2}$, where $v_s\in T_s\S^{k-1}$, $v_t\in T_t[0,\e/2)\simeq\R$. Using Claim~\ref{claim_proj}, the bound~(\ref{eq_AA}), the above bound on the gradient and the obvious inequalities $|v_t|\le \|v\|_\e$, $\e\|v_s\|\le\|v\|_\e$, we obtain
\begin{multline*}
\|D_{s,t}\tilde A^nv\|^2_\e=\e^2\|D_s\hat A^nv_s\|^2+|\AA^n(s)v_t+t<\nabla\AA^n(s),v_s>|^2\\
\le\e^2C^2\left(\frac\nu\tau\right)^{2n}\|v_s\|^2+|c_1\nu^n|v_t|+t\|\nabla\AA^n(s)\|\|v_s\||^2\\
\le C^2\left(\frac\nu\tau\right)^{2n}\|v\|_\e^2+|c_1\nu^n\|v\|_\e+tc_5\AA^n(s)\left(\frac\nu\tau\right)^{n}\|v_s\||^2\\
\le C^2\left(\frac\nu\tau\right)^{2n}\|v\|_\e^2+|c_1\nu^n\|v\|_\e+c_5\frac\e2\left(\frac\nu\tau\right)^{n}\|v_s\||^2\\
\le C^2\left(\frac\nu\tau\right)^{2n}\|v\|_\e^2+|c_1\nu^n\|v\|_\e+c_5\left(\frac\nu\tau\right)^{n}\|v\|_\e|^2
\le c_6\left(\frac\nu\tau\right)^{2n}\|v\|_\e^2.
\end{multline*}
Hence we have established the posited upper bound. The proof of the lower bound takes the same route by rewriting the lower bound as an upper bound on the differential of $\tilde A^{-1}$ and using the same steps. (Note that the main auxiliary bounds~(\ref{eq_AA}) and the bounds in Claim~\ref{claim_proj} are symmetric.) Hence  the proof of Lemma~\ref{lemma_core} is complete modulo Claim~\ref{claim_proj}.
\end{proof}

\begin{proof}[Proof of Claim~\ref{claim_proj}] 
This claim is well-known and easy, however, we couldn't locate a reference in the literature. 

Realize $(\S^{k-1}, \|\cdot\|)$ as the unit sphere in $(\R^k, can)$. Then given $v\in T_s\S^{k-1}$ we can decompose $ v\mapsto D\hat A^nv$ as the following composition
$$
(s,v)\mapsto (A^ns,DA^nv)\mapsto \left(\frac{A^ns}{\|A^ns\|},\frac{DA^nv}{\|A^ns\|}\right)\mapsto (\hat A^ns, D\hat A^nv),
$$
where the first map is self-explanatory, the second is a homothety and the third one is just the projection on the tangent space $T_{\hat A^ns}\S^{k-1}$ (and hence has norm $\le 1$). Hence we have
$$
\|D\hat A^n v\|\le\frac{\|DA^n v\|}{\|A^ns\|}\le C\left(\frac\nu\tau\right)^n.
$$
The proof of the lower bound is analogous.
\end{proof}

\subsection{The proof of partial hyperbolicity}
\subsubsection{The scheme}
The strategy of the proof is fairly straightforward. The stable, the unstable and the center distributions for $\hat f$ --- $\hat E^s$, $\hat E^u$ and $\hat E^c$ ---  away from the exceptional set are pull-backs by the blow-down map $\pi\colon \hat M\to M$ and extend continuously to the exceptional set. It is crucial to consider special Riemannian metrics $\hat g_\e$ on $\hat M$ so that $(\tilde \D^k\times N, g_\e)\subset (\hat M, \hat g_\e)$, $\e\in(0,\e_0)$, are isometric embeddings. The exponential estimates for the action of $D\hat f$ along $\hat E^s$ and $\hat E^u$ are easy and the main difficulty is to control $D\hat f|_{\hat E^c}$ in the neighborhood of the exceptional set $\R P^{k-1}\times N\subset \hat M$. Because the center distribution is close to the ``horizontal" distribution near $\R P^{k-1}\times N$, Lemma~\ref{lemma_core} provides control on $D\hat f|_{\hat E^c}$ in the neighborhood of $\R P^{k-1}\times N$. However, an orbit can return to this neighborhood infinitely often and, hence, the constant $C>0$ of Lemma~\ref{lemma_core} could contribute to the exponential rate. This problem is addressed by letting $\e\to 0$. For smaller $\e$ the orbit would spend larger time outside of the neighborhood of $\R P^{k-1}\times N$ where metric was altered. This implies that the exponential contribution of $C>0$ can be made arbitrarily close to 1 which yields partial hyperbolicity.
\subsubsection{Riemannian metrics and partial hyperbolicity} Recall that we have smoothly identified a neighborhood of $N$  with $\D^k\times N$. Let us equip $M$ with a Riemannian metric $g$ such that the restriction of $g$ to $\D^k\times N$ is the direct sum $g=can+g_N$, where $g_N$ is a Riemannian metric on $N$. Recall that $f\colon M\to M$ is partially hyperbolic and inequalities~(\ref{def_ph}) hold with respect to {\em some} Riemannian metric. For the newly chosen metric $g$ inequalities~(\ref{def_ph}) do not necessarily hold, however $\exists K>0$ and $\exists \delta>0$ such that $\forall n>0$
\begin{equation}
\label{def_ph2}
\begin{split}
&\|Df^n(v^s)\|_g\le K(\lambda-\delta)^n\;\;\;\\
K^{-1}(\lambda+\delta)^n\le& \|Df^n(v^c)\|_g\le K(\mu-\delta)^n  \\
K^{-1}(\mu+\delta)^n\le &\|Df^n(v^u)\|_g
\end{split}
\end{equation}
for all unit vectors $v^s\in E^s$, $v^c\in E^c$ and $v^u\in E^u$.
Note that existence of positive $\delta$ comes from strict inequalities~(\ref{def_ph}) and compactness of $M$.

Now for each $\e\in(0,\e_0)$ equip $\hat M$ with the Riemannian metric $\hat g_\e$ which coincides with $g_\e+g_N$ on $\tilde \D^k\times N$ and with $g$ elsewhere. Note that the blow-down map $\pi\colon (\hat M,\hat g_\e) \to (M, g)$ is an isometry on the complement of $\tilde \D^k_{<\e}\times N$. Denote $\|\cdot\|^2_\e=\hat g_\e(\cdot,\cdot)$. To establish partial hyperbolicity of $\hat f\colon\hat M\to\hat M$ we will show that there exists $\delta>0$ and a $D\hat f$-invariant splitting $T\hat M=\hat E^s\oplus\hat E^c\oplus\hat E^u$, an $\e>0$ and $\hat C>0$ such that  $\forall n>0$
\begin{equation}
\label{eq_final_estimate}
\begin{split}
&\|D\hat f^n(v^s)\|_\e\le \hat C(\lambda-\delta)^n\;\;\;\\
\hat C^{-1}(\lambda+\delta)^n\le& \|D\hat f^n(v^c)\|_\e\le \hat C(\mu-\delta)^n  \\
\hat C^{-1}(\mu+\delta)^n\le &\|D\hat f^n(v^u)\|_\e
\end{split}
\end{equation}
for unit vectors $v^s\in \hat E^s$, $v^c\in \hat E^c$ and $v^u\in \hat E^u$.

\subsubsection{The stable and unstable distributions}

The restriction $\pi\colon\hat M\backslash \R P^{k-1}\times~N\to M\backslash N$ of the blow-down map is a diffeomorphism. Hence away from the exceptional set we can pull back the the stable and unstable distributions
$$
\hat E^s|_{\hat M\backslash \R P^{k-1}\times N}\stackrel{\mathrm{def}}{=} D\pi^{-1} E^s|_{M\backslash N},\,\,\,\,
\hat E^u|_{\hat M\backslash \R P^{k-1}\times N}\stackrel{\mathrm{def}}{=} D\pi^{-1} E^u|_{M\backslash N},
$$ 
Recall that by the locally fiberwise assumption~(\ref{eq_fiberwise2}) distributions $E^s$ and $E^u$ are tangent to the $N$-fibers in the neighborhood $\D^k\times N\subset M$. It follows that $\hat E^s$ and $\hat E^u$ are also tangent the $N$-fibers in the the neighborhood $\tilde\D^k\times N\subset \hat M$. Therefore distributions $\hat E^s$ and $\hat E^u$ extend continuously to the exceptional set $\R P^{k-1}\times N$.

Notice that, by definition of $\hat g_\e$, if $v\in\hat E^s\oplus \hat E^u(x)$, then $\|v\|_\e=\sqrt{g_N(v,v)}$. It immediately follows that ~(\ref{def_ph2}) implies that $\forall n>0$
\begin{equation*}
\begin{split}
&\|D\hat f^n(v^s)\|_\e\le K(\lambda-\delta)^n\;\;\;\\
K^{-1}(\mu+\delta)^n\le &\|D\hat f^n(v^u)\|_\e
\end{split}
\end{equation*}
for all unit vectors $v^s\in \hat E^s$, and $v^u\in \hat E^u$. Hence it remains to establish the middle inequality of~(\ref{eq_final_estimate}).

\subsubsection{The center distribution}
Let $H$ be the ``horizontal" distribution tangent to the $\D^k$-fibers in the neighborhood $\D^k\times N\subset M$ and let $\hat H$ be the ``horizontal" distribution tangent to the $\tilde \D^k$-fibers in the neighborhood $\tilde \D^k\times N\subset \hat M$. By Remark~\ref{rem_center}, $E^c|_N=H|_N$. 

As before, away from the exceptional set define
$$
\hat E^c|_{\hat M\backslash \R P^{k-1}\times N}\stackrel{\mathrm{def}}{=} D\pi^{-1} E^c|_{M\backslash N}.
$$
Because the angle $\angle_g(E^c(x), H(x))\to 0$ as $x$ approaches the exceptional set $N$, we also have that $\angle_{\hat g_\e}(\hat E^c(x), \hat H(x))\to 0$ as $x$ approaches the exceptional set $\R P^{k-1}\times N$. Hence, $\hat E^c$ extends continuously to the exceptional set and
\begin{equation*}
\label{eq_EH}
\hat E^c|_{\R P^{k-1}\times N}=\hat H|_{\R P^{k-1}\times N}.
\end{equation*}

\subsubsection{The local center estimate}
\label{sec_center}
Lemma~\ref{lemma_estimate} provides exponential estimates for the action of $D\hat f$ on $\hat H$. Namely, given a finite orbit $\{x, \hat f x,\ldots \hat f^nx\}\subset \tilde \D^k\times N$ and a vector $v^h\in \hat H(x)$, Lemma~\ref{lemma_estimate} gives
\begin{equation}
\label{eq_h}
C^{-1}(\tau/\nu)^n\|v^h\|_\e\le\|D\hat f^nv^h\|_\e\le C (\nu/\tau)^n\|v^h\|_\e
\end{equation}
 Inequalities~(\ref{eq_nu_tau}) imply that there exists a $\delta_1>0$ such that
$$
\lambda+\delta_1<\frac\tau\nu<\frac\nu\tau<\mu-\delta_1.
$$
Hence~(\ref{eq_h}) implies
\begin{equation}
\label{eq_h2}
C^{-1}(\lambda+\delta_1)^n\|v^h\|_\e\le\|D\hat f^nv^h\|_\e\le C (\mu-\delta_1)^n\|v^h\|_\e
\end{equation}

The goal now is to obtain same estimates for $v^c\in\hat E^c$ near the exceptional set, where $\hat E^c$ is close to $\hat H$.

Pick a small $\omega>0$ and let $x\in(\tilde \D^k_{<\omega}\times N)\backslash(\R P^{k-1}\times N)$. Pick a $v^c\in \hat E^c(x)$ and decompose $v^c=v^h+v^v$, where $v^h\in\hat H(x)$ and $v^v$ is the ``vertical vector" tangent to the $N$-fiber through $x$. We can pull back this splitting of $v^c$ to $T_{\pi^{-1}x}(\D^k\times N)$ using $\pi$ to the splitting $\bar v^c=\bar v^h+\bar v^v$. Then we have
$$
\frac{\|v^v\|_\e}{\|v^h\|_\e}\le\frac{\|\bar v^v\|}{\|\bar v^h\|}\le c_1\omega^\alpha,\,\,\, \alpha\in (0,1),
$$
where the first inequality is by the definition of the metric $g_\e$ and the second one is by H\"older continuity of $E^c$ at $N$ (in fact, any uniform modulus of continuity would be sufficient for further purpose). This estimate on the ratio of ``vertical" and ``horizontal" components of $v^c$ makes it possible to compare the expansion of $v^c$ to that of $v^h$ as follows
\begin{multline}
\label{eq_rates}
\frac{\|D\hat fv^c\|^2_\e}{\|v^c\|^2_\e}\le \frac{\|D\hat fv^h\|^2_\e+\|D\hat fv^v\|^2_\e}{\|v^h\|^2_\e}
\le \frac{\|D\hat fv^h\|^2_\e}{\|v^h\|^2_\e}+c_2 \frac{\|v^v\|^2_\e}{\|v^h\|^2_\e}\\
\le \frac{\|D\hat fv^h\|^2_\e}{\|v^h\|^2_\e}+c_3\omega^{2\alpha}
\le \frac{\|D\hat fv^h\|^2_\e}{\|v^h\|^2_\e}(1+c_4\omega^{2\alpha})
\le \frac{\|D\hat fv^h\|^2_\e}{\|v^h\|^2_\e}(1+\sqrt{c_4}\omega^{\alpha})^2
\end{multline}
The constant $c_2$ is the bound on the expansion of ``vertical" vectors which independent of $\e$ because $\|\cdot\|_\e$ does not depend on $\e$ for ``vertical" vectors. The constant $c_3/c_4$ is the upper bound on the expansion of ``horizontal" vectors, which is independent of $\e$ by~(\ref{eq_bound2}).

Estimate~(\ref{eq_rates}) implies that, provided the orbit stays in $\tilde \D^k_{<\omega}\times N$, we can replace $v^h$ by $v^c$ in~(\ref{eq_h2}) after adjusting the upper bound by a small exponential term
\begin{equation*}
\|D\hat f^nv^c\|_\e\le C (\mu-\delta_1)^n(1+\sqrt{c_4}\omega^{\alpha})^n\|v^c\|_\e
\end{equation*}
Analogous lower bound can be established in the similar way. We conclude that there exists an $\omega>0$ and a $\delta_2>0$ such that for any $\e<\omega$ and any finite orbit $\{x, \hat f x,\ldots \hat f^nx\}\subset \tilde \D^k_{<\omega}\times N$ and any vector $v^c\in \hat E(x)$ we have
\begin{equation}
\label{eq_local_center}
C^{-1}(\lambda+\delta_2)^n\|v^c\|_\e\le\|D\hat f^nv^c\|_\e\le C (\mu-\delta_2)^n\|v^c\|_\e
\end{equation}

\subsubsection{The global center estimate}
First note that~(\ref{eq_h2}) takes care of the posited center estimate~(\ref{eq_final_estimate}) in the case when $x\in\R P^{k-1}\times N$ and $v^h\in \hat E^c(x)=\hat H(x)$. Now we will explain how~(\ref{eq_local_center}) implies the posited center estimate on the complement of the exceptional set. 

Recall that there exists $c=c(A)>1$ such that 
\begin{equation}
\label{eq_e_choice}
\hat f(\tilde \D^k_{\e}\times N)\cup \hat f^{-1}(\tilde \D^k_{<\e}\times N)\subset \tilde \D^k_{<c\e}\times N
\end{equation}
We pick $\omega>0$ so that~(\ref{eq_local_center}) holds and then we consider all $\e\in(0,\frac{\omega}{c})$. 
Cover $\hat M$ by two open sets $\U_\e=\tilde \D^k_{<c\e}\times N$ and $\V_\e=\hat M\backslash \tilde \D^k_{<\e}\times N$. 

Now pick any $x\in \hat M$ which is not in the exceptional set and consider a finite orbit segment $\{x, \hat f x,\ldots \hat f^{n-1}x\}$. This orbit can be partitioned into a finite number of (disjoint) segments $O_{y,i}=\{y, \hat f y,\ldots \hat f^{i}y\}$ such that each segment is entirely contained either in $\U_\e$ or in $\V_\e$ (and the orbit segments alternate between $\U_\e$ and $\V_\e$). Moreover, because of~(\ref{eq_e_choice}), this partition can be chosen so that for each segment $O_{y,i}$ the next point in the orbit $\hat f^{i+1}y$ also belongs to the open set ($\U_\e$ or $\V_\e$) containing $O_{y,i}$. 

Now, given a $v^c\in \hat E^c(x)$ we have the decomposition
 \begin{equation}
 \label{eq_decomp}
 \frac{\|D\hat f^nv^c\|_\e}{\|v^c\|_\e}=\prod_{O_{y,i}}\frac{\|D\hat f^{i+1}v_y^c\|_\e}{\|v^c_y\|_\e}
 \end{equation}
where $v^c_y\in \hat E^c(y)$ is the image of $v^c$ under the appropriate iterate. Recall that we have  estimates for each factor in the product. Namely, if $O_{y,i}\cup\hat f^{i+1}y\subset \U_\e\subset \tilde\D^k_{<\omega}$ then
$$
C^{-1}(\lambda+\delta_2)^n\le\frac{\|D\hat f^{i+1}v_y^c\|_\e}{\|v^c_y\|_\e}\le C (\mu-\delta_2)^n
$$
by~(\ref{eq_local_center}). And if $O_{y,i}\cup\hat f^{i+1}y\subset \V_\e$ then from partial hyperbolicity of $f$~(\ref{def_ph2})  and the fact that $\pi^*g=\hat g_\e$ on $\V$ we have
$$
K^{-1}(\lambda+\delta)^n\le\frac{\|D\hat f^{i+1}v_y^c\|_\e}{\|v^c_y\|_\e}\le K (\mu-\delta)^n.
$$

Both of these estimates have constants ($C$ and $K$) which, of course, will contribute exponentially to the product~(\ref{eq_decomp}). However both constants are independent of $\e$. (Recall that $C$ comes from Lemma~\ref{lemma_estimate}.) Now by sending $\e$ to $0$ we shrink the open neighborhood $\U_\e$ of $N$. The decomposition into orbit segments is, of course, changing. And it follows that once an orbit leaves $\U_\e$ it takes a longer time to return to $\U_\e$ again. Hence, by choosing sufficiently small $\e$, the orbit segments $O_{y,i}\subset \V_\e$ can be made arbitrary long. It follows that the contribution of $C$ and $K$ to the product~(\ref{eq_decomp}) can be ``absorbed" by a small adjustment of the exponential rate of the estimate on these longer pieces in $\V_\e$. (This are very standard inequality manipulations and we suppress the details.) We conclude that there exists $\delta_3>0$ and an $\e>0$ such that for all $ v^c\in\hat E^c$
$$
(CK)^{-1}(\lambda+\delta_3)^n\|v^c\|_\e\le\|D\hat f^nv^c\|_\e\le CK (\mu-\delta_3)^n\|v^c\|_\e
$$
\hfill$\square$
\begin{remark} We would like to point out that localization at the exceptional set played a significant role twice. First,  we had chosen a small $\omega$ so that the linear estimate for the center given by Lemma~\ref{lemma_estimate} yields non-linear estimate along $\hat E^c$ near the exceptional set~(\ref{eq_local_center}). Second, we had to shrink the region $\U_\e$ where the metric $\hat g_\e$ differs from $g$ so that partial hyperbolicity away from $\U_\e$ takes care of contributions of the constants $C$ and $K$. Uniform control on $\hat g_\e$ was, of course, crucial for this argument. Namely, the fact that $C$ is independent of $\e$.
\end{remark}

\subsection{Complex blow-up}
\label{sec_complex}
Here we explain how the proof can be adapted to the case of complex blow-up to yield Theorem~\ref{thm_complex}.
\subsubsection{The Fubini-Study metric and a family of Riemannian metrics on the sphere}
Let $\S^{2k-1}\subset \C^k$ be the unit sphere equipped with the standard round metric $ds^2$. The circle $S^1\subset\C$ acts on $\S^{2k-1}$ by scalar multiplication $e^{i\varphi}\cdot(z_1,z_2,\ldots z_k)\mapsto (e^{i\varphi}z_1,e^{i\varphi}z_2,\ldots e^{i\varphi}z_k)$. This action makes $\S^{2k-1}$ into the total space of the (generalized) Hopf fibration
$$
S^1\to\S^{2k-1}\stackrel{H}{\to}\C P^{k-1}.
$$
Moreover, the $S^1$ action is isometric.

Let $X$ be the one-dimensional distribution tangent to the orbits of $S^1$ action and let $X^\perp$ be the orthogonal distribution. We can decompose $ds^2$ accordingly as
$$
ds^2=d\varphi^2+h,
$$
where $d\varphi^2$ is the metric along the $S^1$-fibers and $h$ is the metric on the orthogonal complement. More precisely, if $pr\colon T\S^{2k-1}\to X$ is the orthogonal projection then
$$
d\varphi^2(v_1,v_2)=ds^2(pr(v_1),pr(v_2))
$$
and $h$ is the difference
$$
h(v_1,v_2)=ds^2(v_1-pr(v_1), v_2-pr(v_2)).
$$
The restriction $h|_{X^\perp}$ is a positive definite symmetric bilinear form. Clearly the $S^1$ action preserves $X^\perp$ and $h|_{X^\perp}$. Hence we can define the {\it Fubini-Study} metric $\bar h$ on $\C P^{k-1}$ by pushing forward $h$
$$
\bar h(DH(v_1),DH(v_2))=h(v_1,v_2),\,\,\, v_1,v_2\in X^\perp.
$$
Also for each $\mu\ge 0$ let
$$
h_\mu=\mu^2d\varphi^2+h.
$$
For $\mu>0$ this yields a Riemannian metric and for $\mu=0$ a degenerate metric with circle fibers of zero length. Note that for each $\mu\ge 0$ the $S^1$ action is isometric and $H\colon (\S^{2k-1}, h_\mu)\to (\C P^{k-1},\bar h)$ is a Riemannian submersion. We refer to~\cite[Sections 1.4 and 2.5]{Petersen} for much more detailed discussion and explicit doubly warped expressions for the Fubini-Study metric.

\subsubsection{A family of Riemannian metrics on $\tilde \D^k_\C$} Similarly to the real case we begin with the ``spherical" blow-up
$$
\bar\D^k_\C=\{ (x,r(x)), x\in \D^k_\C,\, x\in r(x)\}.
$$
Under the identification $\bar \D^k_\C\simeq \S^{k-1}\times[0,1)$, the map $\bar\D^k_\C\to\tilde \D^k_\C$ which sends each real ray $r(x)$ to the unique complex line containing it, becomes the quotient map
\begin{equation}
\label{eq_quotient}
\S^{2k-1}\times[0,1)\to \S^{2k-1}\times[0,1)/\!\!\!\sim
\end{equation}
where the relation $\sim$ is given by the Hopf action of $S^1$ on $\S^{2k-1}\times\{0\}$.

We define a family of metrics $g_\e, \e\in(0,\e_0)$, on $\bar \D^k_\C$ in the following doubly warped form
$$
g_\e\stackrel{\mathrm{def}}{=}dt^2+\rho_\e(t)^2h_{\mu_\e}=dt^2+\rho_\e(t)^2h+(\rho_\e(t)\mu_\e(t))^2d\varphi^2=dt^2+\rho_\e(t)^2h+t^2d\varphi^2,
$$
where $\mu_\e(t)=t\rho_\e(t)^{-1}$ and $\rho_\e$ is a smooth function which satisfies
$$
 \rho_\e(t)=
 \begin{cases} \e, \,\,\,t\le\e/2\\
 t, \,\,\,t\ge\e
 \end{cases}
 $$
 and 
 \begin{equation*}
 t\le \rho_\e(t)\le\e,\,\,\,\,\, t\in[\e/2,\e].
 \end{equation*}
We notice that $g_\e|_{\bar \D^k_{>\e, \C}}$ is the standard Euclidean metric because $h_1=ds^2$. For $t>0$ we clearly have a smooth Riemannian metric. However, when $t=0$, the metric becomes degenerate, namely, $g_\e|_{\S^{2k-1}\times\{0\}}=h_0$. Because $h_0$ is $S^1$-invariant, metrics $g_\e$ factor through to a true Riemannian metrics on $\tilde \D^k_\C$ so that the quotient map~(\ref{eq_quotient}) is an isometry. Abusing the notation, we still denote this family of metrics on $\tilde \D^k_\C$ by $g_\e$. One can check that $g_\e$ is indeed a smooth metric at the exceptional locus $\C P^{k-1}\subset \tilde\D^k_\C$ by using the standard smooth charts for the blow-up, such as
\begin{equation}
\label{eq_chart}
(z_1,z_2,\ldots z_k)\mapsto (z_1, z_1z_2, \ldots z_1z_k, [1:z_2:\ldots :z_k]).
\end{equation}

\subsubsection{Local dynamics near the exceptional set} Now we explain that the metrics $g_\e$ possess local product structure on $\tilde \D^k_{<\e/2,\C}$ and that $\tilde A$ behaves like a skew product with respect to this product structure.

The manifold $\tilde \D^k_{<\e/2,\C}\backslash\C P^{k-1}$ is the product $\S^{2k-1}\times (0,\e/2)$. We have the distributions $X$ and $X^\perp$ on each sphere fiber $\S^{2k-1}\times\{t\}$ and we can define the assembled distribution
$$
E=\bigcup_{t\in(0,\e/2)} X^\perp|_{\S^{2k-1}\times\{t\}}.
$$
Also let
$$
F=\frac{\partial}{\partial t}\oplus \bigcup_{t\in(0,\e/2)} X|_{\S^{2k-1}\times\{t\}}.
$$
Since the splitting $T\S^{2k-1}=X\oplus X^\perp$ is an orthogonal splitting with respect to every metric $h_\mu$ and $g_\e$ has warped form we have that the splitting $T(\tilde \D^k_{<\e/2,\C}\backslash\C P^{k-1})=F\oplus E$ is orthogonal with respect to $g_\e$. This splitting smoothly extends to $\C P^{k-1}$ so that
$$
E|_{\C P^{k-1}}=T\C P^{k-1}\subset T\tilde \D^k_{<\e/2,\C}|_{\C P^{k-1}}
$$
This again, can be seen using charts. For example the splitting $E\oplus F([1,0,\ldots 0])$ when expressed in the chart~(\ref{eq_chart}) becomes $T\C^{k-1}\oplus T\C(0,0,\ldots 0)$, where $\C^{k-1}=\{(z_2, z_3,\ldots z_k)\}$ and $\C=\{(z_1,0,\ldots 0)\}$.

Distribution $F$ integrates to one-complex-dimensional disks $\D^1_{<\e/2,\C}$ and the restriction of $g_\e$ to these disks is given by $dt^2+t^2d\varphi^2$. Hence we can view $\tilde \D^k_{<\e/2,\C}$ as a fiber bundle
$$
 \D^1_{<\e/2,\C}\to\tilde \D^k_{<\e/2,\C}\to\C P^{k-1} 
$$
with flat fibers.
Moreover the projection map $(\tilde \D^k_{<\e/2,\C}, g_\e)\to(\C P^{k-1},\bar h)$ is a Riemannian submersion.

Our next observation is that the induced map $\tilde A\colon \tilde \D^k_{<\e/2,\C}\to \tilde \D^k_{<\e/2,\C}$ preserves $F$. Indeed, by linearity, $A$ preserves the real rays (integral lines of $\frac{\partial}{\partial t}$) and since $A$ is complex-linear it preserves $X$. Hence $\tilde A$ fits into the commutative diagram
\begin{equation}
\label{eq_diagram}
\xymatrix{
\tilde \D^k_{<\e/2,\C}\ar[d]\ar^{\tilde A}[r] & \tilde \D^k_{<\e/2,\C}\ar[d]\\
\C P^{k-1}\ar^{\hat A}[r] & \C P^{k-1}
}
\end{equation}
where $\hat A$ is the complex projectivization of $A\colon \C^k\to\C^k$.
Moreover, $\tilde A$ is conformal on the fibers.\footnote{For the real blow-up the situation was similar. We also had a non-trivial interval bundle over $\R P^{k-1}$, but we had the luxury to pass to the double cover which trivialized the bundle and allowed us to work with a true skew product.}

\subsubsection{The estimates}
The proof of partial hyperbolicity of the diffeomorphism $\hat f\colon \hat M_\C\to\hat M_\C$ follows the steps of the proof of the Main Theorem very closely. In particular, the proof of the second part where the local estimate of Lemma~\ref{lemma_estimate} is used to establish partial hyperbolicity goes through without any alternations at all. 

For the proof of the analogue of Lemma~\ref{lemma_estimate} (relative to the family of metric $g_\e$ constructed above) for $\tilde A\colon\tilde \D^k_\C\to\tilde \D^k_\C$ recall~(\ref{eq_decomp}) that we have partitioned the finite orbit into 5 orbit segments according to the distance to the exceptional set. Because $g_\e$ is flat on $\tilde \D^k_{>\e,\C}$ and the transition domain $\tilde \D^k_{[\e/2,\e],\C}$ contains only uniformly bounded number of points from the orbit, the exact same argument which we have used for the proof of Lemma~\ref{lemma_estimate}, works again here. Hence we only need to look at the domain $\tilde \D^k_{<\e/2,\C}$ where the metric $g_\e$ is different in the complex case. Namely, given a finite orbit  $\{x, \tilde Ax, \tilde A^2x,\ldots \tilde A^{n}x\}\subset \DD^k_{<\e/2,\C}$ and any $v\in T_x\tilde \D^k_{<\e/2,\C}$ we need to show that there exists a $C>0$ (which does not depend on $\e$) such that for all $n>0$
$$
C^{-1}(\tau/\nu)^n\|v\|_\e\le\|D\tilde A^nv\|_\e\le C (\nu/\tau)^n\|v\|_\e
$$
The proof of this bound follows the proof of Lemma~\ref{lemma_core} making use of the structure of $g_\e$ on $\tilde \D^k_{<\e/2,\C}$ on which we have elaborated above. Indeed, the bound on projectivization
$$
C^{-1}(\tau/\nu)^n\|v\|_{\bar h}\le\|D\hat A^nv\|_{\bar h}\le C (\nu/\tau)^n\|v\|_{\bar h},
$$
follows from the Claim~\ref{claim_proj} and the fact that $\hat A \colon \C P^k\to\C P^k$ is the quotient of $\hat A\colon \S^{2k-1}\to \S^{2k-1}$ by the Riemannian submersion $(\S^{2k-1}, ds^2)\to (\C P^{k-1},\bar h)$. Further, the function $s\mapsto \|As\|/\|s\|$, $s\in\S^{2k-1}\subset\C^k$, factors through to a function $a\colon\C P^{k-1}\to \R$ which generates a cocycle $\AA^n\colon\C P^{k-1}\to\R$ which is controlled by $\tau^n$ and $\nu^n$~(\ref{eq_AA}).

Finally we make use of the skew product structure~(\ref{eq_diagram}) (just as we did in the real case) to establish the posited estimates. Namely, given a $v\in T_x \tilde \D^k_{<\e/2,\C}$ decompose $v=v_E+v_F$, $v_E\in E(x)$, $v_F\in F(x)$. Then growth of $v_F$ is controlled by the bounds on the cocycle $\AA^n$ and the growth of the $E$-component of $v_E$ is controlled by the bounds on $\hat A^n$. Since $E$ is not $\tilde A$-invariant $v_E$-component also yields some ``shear growth" which can be controlled, just as in the proof of Lemma~\ref{lemma_estimate}, by estimating the gradient $\nabla\AA^n$.

\medskip
\medskip
\medskip
Andrey Gogolev\\
SUNY Binghamton, N.Y., 13902, USA\\

\end{document}